\documentclass{amsart}

\usepackage{amsthm}
\usepackage{amsmath}
\usepackage{amsfonts}
\usepackage{amssymb}
\usepackage{graphicx}
\usepackage{enumerate}
\usepackage{amsmath}
\usepackage{amsfonts}
\usepackage{textcomp}
\usepackage{listings}
\usepackage{multicol}
\usepackage{epigraph}
\usepackage{mathtools}
\usepackage{booktabs}
\usepackage{tikz-cd}
\usepackage[all]{xy}
\usepackage{eqnarray}
\usepackage[hidelinks]{hyperref}
\usepackage{faktor}
\usepackage[export]{adjustbox}
\usepackage{float}
\usepackage{pdfpages}
\usepackage[mathscr]{euscript}

\newtheorem{thm}{{\rm T\sc heorem}}[section]
\newtheorem{cor}[thm]{\rm C\sc orollary}

\newtheorem{lem}[thm]{\rm L\sc emma} 
\newtheorem{prop}[thm]{\rm P\sc roposition}
\newtheorem*{con}{\rm C\sc onjecture}

\theoremstyle{definition}
\newtheorem{defn}[thm]{{\rm D\sc efinition}}
\newtheorem{rmk}[thm]{{\rm R\sc emark}}
\newtheorem{ex}[thm]{\rm E\sc xample}
\newtheorem*{acks}{Acknowledgments}
\newtheorem*{notation}{Notation}

\begin{document}
\title {A nonspecial divisor in the moduli space of cubic fourfolds via 10-nodal plane sextics}
\author[Elena Sammarco]{Elena Sammarco}
\address{Dipartimento di Matematica e Fisica \\ 
Università degli studi Roma Tre}
\email{elena.sammarco@uniroma3.it}

\begin{abstract}
    In the moduli space $\mathcal C$ of complex cubic hypersurfaces $X\subset\mathbb P^5$, we study the condition that $X$ admits a net of polar quadrics whose discriminant locus is a $10$-nodal irreducible plane sextic curve. Our main result is that such a condition defines an irreducible divisor in $\mathcal C$ which is not of Noether-Lefschetz type.
\end{abstract}

\maketitle

\pagestyle{myheadings}\markboth{\textsc{Elena Sammarco}}{\textsc{A nonspecial divisor in the moduli space of cubic fourfolds via 10-nodal plane sextics}}

\section{Introduction}

    The study of the Hodge loci in the moduli space of complex cubic hypersurfaces of $\mathbb P^n$ is certainly a fundamental theme when addressing these varieties and their related problems (see \cite{Voisin13}). This is true, in particular, for the moduli space $\mathcal C$ of smooth cubic fourfolds $X\subset\mathbb P^5$; the study in $\mathcal C$ of the geography of its Noether-Lefschetz divisors, and their description, is a central issue. We recall from \cite{Hassett00} that:

    \begin{defn}
        A cubic fourfold $X$ is \textit{special} if it admits an algebraic cycle of codimension 2 not homologous to a complete intersection.
    \end{defn}

    This is equivalent to say that $X$ is not special if and only if its Néron-Severi lattice $H^{2,2}(X,\mathbb{Z})$ is generated only by the class of $h^2$, where $h=c_1(\mathcal{O}_X(1)) \in {\rm Pic } \ X$. It is well known, for example it follows from \cite{Voisin86}, that a very general cubic fourfold $X$ is not special. Moreover, the family of special cubic fourfolds is a countable union of irreducible divisors $\mathcal{C}_d$, $d\in\mathbb{Z}_{\geq 0}$, in the moduli space $\mathcal{C}$. Here $d$ is the discriminant of the lattice $H^{2,2}(X,\mathbb{Z})$, which has rank 2 and constant discriminant $d \ \equiv \ 0 \ \text{or} \ 2\mod 6$ for a general $X$ in $\mathcal C_d$. The divisors $\mathcal C_d$ are called \textit{special} or of \textit{Noether-Lefschetz type}. The Hodge locus of $\mathcal C$ is their union. A special cubic fourfold $X$ often contains some unexpected and beautiful families of algebraic surfaces, not homologous to a complete intersection and useful to describe it geometrically. Moreover, the family of divisors $\mathcal C_d$ conjecturally includes the locus in $\mathcal C$ parametrizing rational cubic fourfolds: this is predicted by a conjecture due to Kuznetsov.

    \begin{con}[Kuznetsov]
        A cubic fourfold $X$ is rational if and only if $[X]$ belongs to some special divisor $\mathcal{C}_d$, with $d$ admissible.
    \end{con}

    The admissibility of $d$ is a very precise numerical condition: $d$ is not divisible by 4, 9 or any odd prime congruent to 2 modulo 3. Now, coming to the rationality problem for $X$, Katzarkov, Kontsevich, Pantev and Yu in their preprint \cite{KKPY25} proved that the very general cubic fourfold is irrational applying their new theory of Hodge atoms. Nevertheless, to detect and describe in $\mathcal C$ the loci of rational cubic fourfolds remains a difficult problem, governed by Kutzentsov's conjecture. Also, as of today, no explicit cubic fourfold is known to be irrational.

    It is implicit, in the previous discussion, the complementary convenience of  investigating other irreducible divisors in $\mathcal C$. These should be nonspecial, but of special interest in order to further understand the birational geometry of a cubic fourfold $X$ and the properties of the moduli space $\mathcal C$. In particular, to our knowledge, only three nonspecial divisors are known and described in the literature and they are due to Ranestad and Voisin \cite{RanestadVoisin17} and to Addington and Auel \cite{AddingtonAuel18}.

    In this work we construct an integral divisor in $\mathcal{C}$, that we call \textit{Severi divisor}. To realize it we use the theory of polar hypersurfaces, the classical Dixon's lemma, generalized to nodal plane curves, and the geometry of some Severi varieties of nodal plane curves. The Severi divisor is the locus defined by the family of cubic fourfolds $X$ admitting a net of polar quadrics whose discriminant is an integral 10-nodal plane sextic. Equivalently, the Hessian hypersurface of $X$ admits an integral $10$-nodal plane section. Then, we prove the following theorem.

    \begin{thm}
        The Severi divisor is not a Noether-Lefschetz divisor.
    \end{thm}
    
    The proof of the theorem follows the idea developed in \cite{RanestadVoisin17} and is based on the definition and the properties of the Hodge loci, as described in \cite{Voisin13}. All the constructions contained in this work are performed on the base field $\mathbb{C}$. This paper is based on the author's PhD thesis \cite{Sammarco25}.

    \begin{acks}
        I want to thank my PhD thesis advisor Alessandro Verra for many fruitful conversations and geometrical discussions having inspired this paper.
    \end{acks}

\section{The Dixon's lemma}\label{thedixonlemma}

\subsection{Theta characteristics on smooth irreducible curves}\label{thetacharacteristics}

    In this section we introduce the notion of theta characteristics and some of their properties which will be useful for the main results of this paper. In particular, we start with smooth irreducible curves.

    \begin{defn}
        Let $C$ be a smooth irreducible curve and $\omega_C$ its canonical sheaf. A \textit{theta characteristic} on $C$ is a line bundle $\theta$ on $C$ such that $\theta\otimes \theta=\theta^{\otimes 2}\cong\omega_C$. A theta characteristic is said to be \textit{even} (respectively, \textit{odd}) if $h^0(C,\theta)\equiv 0 \ (\mathrm{mod} \ 2)$ (respectively, $h^0(C,\theta)\equiv 1 \ (\mathrm{mod} \ 2)$). The pair $(C,\theta)$ is referred to as a \textit{smooth spin curve}.
    \end{defn}

    A theta characteristic is, then, the square root of the canonical line bundle on the curve $C$. Let $g=g(C)$ denote the genus of $C$: then $\theta\in\mathrm{Pic}^{g-1}(C)$.

    \begin{rmk}
        It is a classical result that the number of non isomorphic theta characteristics on a fixed curve $C$ is equal to $2^{2g}$. More precisely, there are $2^{g-1}(2^g+1)$ even and $2^{g-1}(2^g-1)$ odd theta characteristics.
    \end{rmk}

    Let $\mathcal{S}_g$ be the moduli space of smooth spin curves of genus $g\geq 2$. It has been classically known that the natural map $\chi:\mathcal{S}_g\longrightarrow\mathcal{M}_g$, viewed as a morphism of Deligne-Mumford stacks, where $\mathcal{M}_g$ is the moduli space of smooth curves of genus $g\geq 2$, is étale and finite of degree $2^{2g}$. Mumford \cite{Mumford71} and Atiyah \cite{Atiyah71} proved that the parity of a spin curve is locally constant in families. As a consequence, the moduli space $\mathcal{S}_g$ splits into the disjoint union of two connected components $\mathcal{S}_g^-$ and $\mathcal{S}_g^+$, where the former parametrizes odd theta characteristic and the latter the even ones. Then there are the two restriction maps $\chi_-:\mathcal{S}_g^-\longrightarrow\mathcal{M}_g$ and $\chi_+:\mathcal{S}_g^+\longrightarrow\mathcal{M}_g$, both étale covers respectively of degree $2^{g-1}(2^g-1)$ and $2^{g-1}(2^g+1)$.

    \begin{defn}\label{smoothplanespincurves}
        From now on, let $\mathcal{P}_d\subset\mathcal{M}_g$ denote the moduli space of smooth plane curves of degree $d$, where $g=\frac{(d-1)(d-2)}{2}$ is the genus of a smooth plane curve of degree $d$. Consider the étale finite cover $\chi:\mathcal{S}_g\longrightarrow\mathcal{M}_g$; then, we can restrict to the moduli space $\mathcal{P}_d$ and consider
        \[
            \mathcal{PS}_d:=\chi^{-1}(\mathcal{P}_d)\subset\mathcal{S}_g.
        \]
        It is the moduli space of smooth plane spin curves of degree $d$. In particular, the restricted map $\chi:\mathcal{PS}_d\longrightarrow\mathcal{P}_d$ is still an étale finite cover of degree $2^{2g}$. We denote by
        \[
            \mathcal{PS}_d^+:=\mathcal{PS}_d\cap\mathcal{S}_g^+
        \]
        the moduli space of smooth plane even spin curves of degree $d$.
    \end{defn}

\subsection{Compactification of $\mathcal{S}_g$}\label{compactificationofspincurves}

    Mumford's approach to the study of theta characteristics (see \cite{Mumford71}) opened up the way to constructing a proper Deligne-Mumford moduli space $\overline{\mathcal{S}}_g$ of stable spin curves, that was carried out by Cornalba in \cite{Cornalba89}. The moduli space $\overline{\mathcal{S}}_g$ compactifies the moduli space ${\mathcal{S}}_g$ of smooth spin curves, over $\mathbb{C}$. Points in the boundary of this compactification correspond to certain line bundles on nodal curves and lie naturally over points in the boundary of $\overline{\mathcal{M}}_g$. Also, there exist $\overline{\mathcal{S}}_g^+$ and $\overline{\mathcal{S}}_g^-$, the irreducible disjoint moduli spaces of even and odd stable spin curves.

    \begin{defn}[from \cite{Cornalba89}]\label{stablespincurve}
        Let $X$ be a quasistable curve of arithmetic genus $g$. A \textit{spin structure} on $X$ is the datum of a line bundle $\theta$ of degree $g-1$ on $X$ and a homomorphism $\beta:\theta^{\otimes 2}\longrightarrow\omega_X$, satisfying the following conditions:
        \begin{itemize}
            \item $\theta$ has degree 1 on every exceptional component $E$ of $X$, that is, $\theta_{|E}=\mathcal{O}_E(1)$;
            \item $\beta$ is not zero at a general point of every nonexceptional component of $X$.
        \end{itemize}
        A \textit{stable spin curve} of genus $g$ is a triple $(X,\theta,\beta)$, where $X$ is a quasistable curve of genus $g$ and $(\theta,\beta)$ is a spin structure on $X$.
    \end{defn}

    In \cite{Cornalba89}, Cornalba proved the following.
    
    \begin{prop}
        Stable spin curves form the moduli space $\overline{\mathcal{S}}_g$, which is a normal and projective variety, equipped with a regular stabilization morphism
        \[
        \chi:\overline{\mathcal{S}}_g\longrightarrow\overline{\mathcal{M}}_g,
        \]
        given by $\chi([X,\theta,\beta]):=[C]$, where $C$ is the stable model of $X$. $\overline{\mathcal{S}}_g$ contains the moduli space $\mathcal{S}_g$ of smooth spin curves as a dense subvariety. The map $\chi$ is a ramified cover of degree $2^{2g}$.
    \end{prop}

    For a thorough treatment of the boundary of $\overline{\mathcal{S}}_g$ we refer the reader to \cite{Farkas12}.

\subsection{The Dixon's lemma}

    The representation of curves and surfaces of small degree as linear determinants is a classical subject. In 1902 Dixon, in \cite{Dixon02}, posed the problem of expressing the equation of a plane curve of degree $d$ in the form of a symmetric determinant. In particular, he proved that any smooth plane curve $C\subset\mathbb{P}^2$ of degree $d$ can be written as the vanishing of the determinant of a $d\times d$ symmetric matrix whose entries are linear forms in the three variables of $\mathbb{P}^2$. Such a matrix can be seen as a bilinear form, that is, a quadric hypersurface in $\mathbb{P}^{d-1}$; more precisely, since the matrix depends on three parameters, it is a quadric moving in a net of quadrics. Any smooth plane curve $C$ of degree $d$ is then the discriminant curve of a net of quadrics in $\mathbb{P}^{d-1}$, that is, the locus parametrizing the degenerate elements of the net. This result is also known as the \textit{Dixon's lemma}. In \cite{Beauville00} Beauville formalizes this result.

    \begin{thm}[{\cite[Proposition 4.2]{Beauville00}}]\label{minimalresolution}
        Let $C$ be a smooth plane curve of degree $d$, defined by an equation $F=0$, and $\theta$ an even theta characteristic on $C$. Then $\theta$ admits a minimal resolution
        \[
            0\longrightarrow\mathcal{O}_{\mathbb{P}^2}(-2)^d\overset{M}{\longrightarrow}\mathcal{O}_{\mathbb{P}^2}(-1)^d\longrightarrow\theta\longrightarrow 0,
        \]
        where the matrix $M\in \mathbf{M}_{d\times d}(\mathbb{C}[x_0,x_1,x_2])$ is symmetric, linear and $\mathrm{det}M=F$. Conversely, the cokernel of such a symmetric matrix $M$ is a theta characteristic $\theta$ on $C$ with $h^0(C,\theta)=0$.
    \end{thm}

    Since the notion of theta characteristic is extended to stable curves, giving the compactification of $\mathcal{S}_g$, it is natural to ask whether the Dixon's lemma can also be extended to plane stable curves. A positive answer comes from \cite{Barth77}, \cite{Beauville77} and \cite{Dolgachev12}. Furthermore, the next proposition follows from the cited works (especially from \cite[Proposition 6.23]{Beauville77} and \cite[Theorem 4.1.4]{Dolgachev12}).

    \begin{prop}\label{dixonlemma}
        Let $C$ be a stable plane curve of degree $d$ and genus $g$ defined by one homogeneous equation $F(x_0,x_1,x_2)=0$ and let $\theta$ be an even theta characteristic on $C$. Then, up to projective transformations, there exists a unique symmetric matrix $M_{\theta}\in\mathbf{M}_{d\times d}(\mathbb{C}[x_0,x_1,x_2])$ associated to the pair $[(C,\theta)]\in\overline{\mathcal{S}}_g$, whose entries are linear forms in $x_0,x_1,x_2$ and such that $\mathrm{det}M_{\theta}=F$.
    \end{prop}

    Such a matrix $M_{\theta}$ represents a net of quadrics in $\mathbb{P}^{d-1}$. Hence, there is a commutative diagram
    \[\begin{tikzcd}
        \mathcal{N}_{d-1}\arrow[dd, "\nu"']\arrow[rr, "\psi"] & &\overline{\mathcal{P}}_d\\
        & & &\psi=\chi_+\circ\nu\\
        \overline{\mathcal{PS}}_d^+\arrow[uurr, "\chi_+"'] & &
    \end{tikzcd}\]
    where $\mathcal{N}_{d-1}:=Gr(3,H^0(\mathcal{O}_{\mathbb{P}^{d-1}}(2)))/\!\!/PGL(d)$ is the moduli space of nets of quadrics of $\mathbb{P}^{d-1}$. The map $\nu$, which associates to a net its discriminant curve with the corresponding even theta characteristic, is surjective and finite of degree 1, by the previous proposition. Then, the map $\psi$ is surjective and finite of degree $2^{g-1}(2^g+1)$, where $g=\frac{(d-1)(d-2)}{2}$.

\section{Construction of the Severi divisor}

\subsection{The polar linear system of a cubic fourfold}

    For the purpose of this work, we fix coordinates $(x_0:\ldots:x_5)$ on $\mathbb{P}^5$, that is, a basis of the 6-dimensional vector space $V_6:=H^0(\mathcal{O}_{\mathbb{P}^5}(1))$. Let $X\subset\mathbb{P}^5$ be a cubic fourfold, defined by a nonzero homogeneous cubic polynomial $F\in\mathrm{Sym}^3V_6$. Moreover, $X\in\mathbb{P}^{55}:=|\mathcal{O}_{\mathbb{P}^{5}}(3)|$, the linear system of all the cubic fourfolds of $\mathbb{P}^5$. Consider the vector space of polar forms of $X$,
    \[
        J_X:=\left\{\sum_{i=0}^{5} p_i\frac{\partial F}{\partial x_i} \ | \ p=(p_0,\dots,p_{5})\in\mathbb{C}^{6}\right\},
    \]
    whose dimension is at most 6. In the sequel we always assume that $\mathrm{dim}J_X=6$, which is for instance the case of any smooth $X$. The associated \textit{polar linear system} is \mbox{$|J_X|:=\mathbb{P}(J_X)\cong\mathbb{P}^{5}$}. Each of its elements is defined by an equation
    \[
        p_0\frac{\partial F}{\partial x_0}+\ldots+p_5\frac{\partial F}{\partial x_5}=0,
    \]
    where $p=[p_0:\ldots:p_5]$ is a point moving in $\mathbb{P}^5=|J_X|$. Such an equation defines a quadric hypersurface in $\mathbb{P}^5$, which is a \textit{polar quadric} of $X$.

\subsection{Cubic fourfolds with the same polar net}

    Consider the nets of polar quadrics of a cubic fourfold $X$ satisfying $\mathrm{dim}J_X=6$: they are the projectivized of all the possible 3-dimensional vector subspaces of $J_X$. For any $X$ we consider the 9-dimensional Grassmannian $Gr(3,J_X)\cong Gr(3,6)$. An element $W\in Gr(3,J_X)$ is a vector subspace of $J_X$, generated by three independent quadratic forms of $J_X$, with parameters $\lambda,\mu,\nu$. We denote by $N:=\mathbb{P}(W)\cong\mathbb{P}^2$ the associated net of quadrics, with $N\subset|J_X|\subset|\mathcal{O}_{\mathbb{P}^{5}}(2)|$. A general point of $N$ parametrizes one polar quadric $Q_N$ whose coefficient matrix $M(Q_N)$ is symmetric of dimension $6\times 6$ with entries that are linear forms in the variables $(\lambda,\mu,\nu)$. The degenerate quadrics of the net are parametrized by the plane curve $\Delta_N=\{\det M(Q_N)=0\}$ of degree 6 contained in \mbox{$N\cong\mathbb{P}^2_{[\lambda:\mu:\nu]}$}. In particular, if $\Delta_5\subset|\mathcal{O}_{\mathbb{P}^5}(2)|$ is the discriminant hypersurface parametrizing the singular quadrics, i.e. the ones with rank $\leq 5$, then $\Delta_N=N\cap\Delta_5$.

    \begin{prop}\label{smoothsextic}
        Let $X$ be a general smooth cubic fourfold. A general net $N$ of polar quadrics of $X$ has a smooth discriminant plane sextic curve $\Delta_N$.
    \end{prop}
    \begin{proof}
        Let us consider the linear system $|\mathcal{O}_{\mathbb{P}^5}(2)|\cong\mathbb{P}^{20}$. It is known that $\mathrm{Sing}(\Delta_5)=\Delta_4$ is the locus parametrizing the quadrics with rank $\leq 4$ and it has dimension 17. Let $X$ be a general smooth cubic fourfold: its polar linear system $|J_X|\cong\mathbb{P}^5$ is thus a general 5-dimensional linear subspace of $\mathbb{P}^{20}$ and the discriminant hypersurface in $|J_X|$ is the intersection $\Delta_5\cap|J_X|=:\Delta_5^X$. The singular locus of $\Delta_5^X$ is given by $\Delta_4\cap|J_X|=:\Delta_4^X$ and has dimension 2. By a dimension count, a general net $N\subset|J_X|$ does not intersect the singular locus $\Delta_4^X$: indeed, $\Delta_4^X$ has codimension 3 in $|J_X|$ while $N$ has dimension 2. Then, $N$ just meets, by Bertini's theorem, smoothly and transversally the smooth locus of $\Delta_5^X$ in the sextic curve $\Delta_N$.
    \end{proof}

    \begin{notation}
        From now on we fix our notation as follows:
        \begin{itemize} 
            \item  $\mathring{\mathbb P}^{55}:=\{X\in\mathbb P^{55} \ | \ \dim J_X=6\}$. This set is open, since its complement is the projectivized dependency locus, in the space $H^0(\mathcal O_{\mathbb P^5}(3))$, of the differential operators $\frac{\partial}{\partial x_0},\dots,\frac{\partial}{\partial x_5}$. Also, it contains the open set of smooth cubic fourfolds.
            \item  $\mathbb G$ for the Grassmannian $Gr(3, H^0(\mathcal O_{\mathbb P^5}(2)))$ of all the nets of quadrics of $\mathbb P^5$.
            \item  $\mathbb G_X$ for the Grassmannian $Gr(3, J_X)$, when $X \in \mathring {\mathbb P}^{55}$. 
        \end{itemize}
    \end{notation}
    
    We fix the same notation for their Pl\"ucker embeddings. We have a natural commutative diagram
    \[\begin{tikzcd}
        \mathbb P^{19}\arrow[hookrightarrow,r] & \mathbb P^{\binom{21}{3} -1}\\
        \mathbb{G}_X\arrow[hookrightarrow,u]\arrow[hookrightarrow,r] & \mathbb{G}\arrow[hookrightarrow,u]
    \end{tikzcd}\]
    where the vertical arrows denote the respective Pl\"ucker embeddings and the horizontal arrows are the induced natural inclusions. The diagram follows from the inclusion of $J_X$ in $H^0(\mathcal O_{\mathbb P^5}(2))$ as a vector subspace. A first simple question is whether every $N \in \mathbb G$ is a polar net of some cubic fourfold $X$. Let us define
    \[
        \widetilde {\mathbb G} := \overline { \{ (X, N) \in \mathring {\mathbb P^{55}} \times \mathbb G \ \vert \ N \in \mathbb G_X \}},
    \]
    with the closure in the product $\mathbb P^{55} \times \mathbb G$. Then, $\dim \widetilde {\mathbb G} = \dim \mathbb P^{55}+\dim \mathbb G_X = 64$ and $\dim \mathbb G = 54$. Moreover, let us consider the two projections
    \[\begin{tikzcd}
        & \widetilde {\mathbb G} \arrow[dl,"\widetilde{\gamma}"']\arrow[dr, "\widetilde{\tau}"] &\\
        \mathbb{P}^{55} & & \mathbb{G}.
    \end{tikzcd}\]

    \begin{thm}\label{10dimensionalfiber}
        The morphism $\widetilde{\tau}: \widetilde {\mathbb G} \longrightarrow \mathbb G$ is surjective.
    \end{thm}
    \begin{proof}
        Since $\dim \widetilde{\mathbb G} - \dim \mathbb G = 10$, it suffices to produce one cubic fourfold $X$ and one polar net $N \in \vert J_X \vert$ such that the fiber $\widetilde{\tau}^{-1}(N)$ is $10$-dimensional. This is made precise in the forthcoming example.
    \end{proof}
    \begin{ex}\label{example}
        Consider three quadrics in $\mathbb{P}^5$, such as the following:
        \begin{itemize}
            \item $Q_0=\{3x_0^2+2x_0x_5+3x_1^2+4x_1x_4+16x_2x_3+4x_2x_5+x_3x_5+6x_1x_5=0\}$
            \item $Q_1=\{2x_0x_1+\frac{4}{3}x_0x_4+2x_0x_5+2x_1x_5+3x_2^2+3x_3^2+2x_3x_4+6x_4x_5+3x_2x_3=0\}$
            \item $Q_2=\{\frac{16}{9}x_0x_3+\frac{4}{9}x_0x_5+2x_1x_2+x_1x_3-2x_2x_5+3x_4^2+3x_5^2=0\}$
        \end{itemize}
        They generate the vector space $W=\{\lambda Q_0+\mu Q_1+\nu Q_2:\lambda,\mu,\nu\in\mathbb{C}\}$ and the associated net of quadrics $N=\mathbb{P}(W)\cong\mathbb{P}^2_{[\lambda:\mu:\nu]}$. Up to a change of coordinates, we want to find all the possible cubic fourfolds $X=V(F)$ such that $W\in\mathbb{G}_X$. We proceed as follows: to find the equation $F$ we impose, without loss of generality, that
        \[\begin{cases}
            \frac{\partial F}{\partial x_0}=3x_0^2+2x_0x_5+3x_1^2+4x_1x_4+16x_2x_3+4x_2x_5+x_3x_5+6x_1x_5 \\
            \frac{\partial F}{\partial x_1}=2x_0x_1+\frac{4}{3}x_0x_4+2x_0x_5+2x_1x_5+3x_2^2+3x_3^2+2x_3x_4+6x_4x_5+3x_2x_3 \\
            \frac{\partial F}{\partial x_2}=\frac{16}{9}x_0x_3+\frac{4}{9}x_0x_5+2x_1x_2+x_1x_3-2x_2x_5+3x_4^2+3x_5^2 
        \end{cases}\]
        from which we deduce that $F=F(x_0,\dots,x_5)$ must assume all the following forms:
        \[\begin{cases}
            \begin{aligned}
                F=&a\left(x_0^3+x_0^2x_5+3x_0x_1^2+4x_0x_1x_4+16x_0x_2x_3+4x_0x_2x_5+x_0x_3x_5+6x_0x_1x_5\right)+\\
                &G_0(x_1,x_2,x_3,x_4,x_5)
            \end{aligned}\\
            \begin{aligned}
                F=&b\left(x_0x_1^2+\frac{4}{3}x_0x_1x_4+2x_0x_1x_5+x_1^2x_5+3x_1x_2^2+3x_1x_3^2+2x_1x_3x_4+6x_1x_4x_5+3x_1x_2x_3\right)+\\
                &G_1(x_0,x_2,x_3,x_4,x_5)
            \end{aligned}\\
            \begin{aligned}
                F=&c\left(\frac{16}{9}x_0x_2x_3+\frac{4}{9}x_0x_2x_5+x_1x_2^2+x_1x_2x_3-x_2^2x_5+3x_2x_4^2+3x_2x_5^2\right)+\\
                &G_2(x_0,x_1,x_3,x_4,x_5)
            \end{aligned}
        \end{cases}\]
        where $a,b,c\in\mathbb{C}-\{0\}$ and the $G_i$'s are homogeneous cubic forms. Then we see that, necessarily, $3a=b$ and $3b=c$. It follows that the equation of the fourfolds we are looking for has the following form
        \[\begin{aligned}
            F(x_0,\dots,x_5)=&ax_0^3+ax_0^2x_5+3ax_0x_1^2+4ax_0x_1x_4+16ax_0x_2x_3+4ax_0x_2x_5+ax_0x_3x_5+ \\
            &6ax_0x_1x_5+3ax_1^2x_5+9ax_1x_2^2+9ax_1x_3^2+6ax_1x_3x_4+18ax_1x_4x_5+ \\
            &9ax_1x_2x_3-9ax_2^2x_5+27ax_2x_4^2+27ax_2x_5^2+C(x_3,x_4,x_5)=0,
        \end{aligned}\]
         where $C(x_3,x_4,x_5)$ is a homogeneous polynomial of degree 3 in the variables $x_3,x_4,x_5$ only. Such a writing represents a family of cubic forms in the variables $x_0,\dots,x_5$ depending on $a\in\mathbb{C}-\{0\}$ and on the 10 parameters of $C(x_3,x_4,x_5)$. Hence, it is a family of cubic fourfolds of projective dimension 10. \qed
    \end{ex}

    The previous arguments prove the next theorem.

    \begin{thm}
        A general $N\in\mathbb{G}$ is a polar net of a general smooth cubic fourfold $X$.
    \end{thm}

    \begin{rmk}\label{importantrmk}
        Over an open neighborhood of the net $N$ of the example \ref{example}, the morphism $\widetilde{\tau}$ is proper and smooth, in particular flat, with fiber $\mathbb P^{10}$. The description of $\widetilde{\tau}$ can be even more accurate, as in the next theorem.
    \end{rmk}
    
    \begin{thm}\label{P10bundle}
        $\widetilde{\tau}$ is a locally trivial $\mathbb P^{10}$-bundle over an open dense set of $\mathbb G$.
    \end{thm}
    \begin{proof}
        The example \ref{example} implies that, on an open dense set $A \subset \mathbb G$, the corresponding family of fibers of $\widetilde{\tau}$ is a family of $10$-dimensional subspaces of $\mathbb P^{55}$. For each $N \in A$, we have indeed that $\widetilde{\tau}^{-1}(N) = \widetilde {\mathbb G} \cdot \left(\mathbb P^{55} \times \{ N \}\right)$ is a $10$-dimensional linear system of cubic fourfolds. Then, by Grauert's theorem, $\widetilde{\tau}_*(\widetilde{\gamma}^* \mathcal O_{\mathbb P^{55}}(1))$ is a vector bundle over $A$, via $\widetilde{\tau}$, and its projectivization is $\widetilde{\tau}^{-1}(A)$.
    \end{proof}

\subsection{The Severi varieties of plane sections of $\Delta_5$}\label{severivarietiesofplanesections}

    The purpose of this section is to study the stratification, by the number of their nodes, of the family of nodal plane sections of the discriminant hypersurface $\Delta_5 \subset |\mathcal{O}_{\mathbb{P}^5}(2)|$. We will also focus on the induced stratification of the plane sections of $\Delta_5^X\subset \vert J_X \vert$, where $X$ is a smooth cubic fourfold. Clearly, the strata of such a stratification are $PGL(6)$-invariant and useful to study various loci in $\mathcal C$. Actually, the core of this work is the study of the locus of points $[X]\in\mathcal{C}$ such that the Severi variety of $10$-nodal irreducible plane sections of $\Delta_5^X$ is not empty. This locus of $\mathcal C$ has virtual codimension one: we will show that it is an irreducible divisor and that it is not of Noether-Lefschetz type.
    
    From now on, we will simply denote by $\mathbb{P}^{20}$ the linear system $|\mathcal{O}_{\mathbb{P}^5}(2)|$. To avoid ambiguous notation, we denote by $n$ the point of $\mathbb G$ defined by the net $N$. Then we consider the universal plane $u: \mathbb U \longrightarrow \mathbb G$, whose fiber at $n$ is $\mathbb{U}_n=N\cong\mathbb{P}^2$. Notice that $\mathbb U$ is naturally embedded as $\mathbb U =\{ (n,q) \in \mathbb G \times \mathbb P^{20} \ \vert \ q \in N \}$ and it is endowed with the projection maps
    \[\begin{tikzcd}
        & \mathbb U\arrow[dl,"u"']\arrow[dr, "t"] &\\
        \mathbb{G} & & \mathbb{P}^{20}.
    \end{tikzcd}\]
    The fiber of $t$ at $q\in\mathbb{P}^{20}$ is the family of the nets $N$ containing the point $q$. It is easy to see that this is biregular to $Gr(2,20)$, that is, the Grassmannian of lines of $\mathbb P^{19}$. Now we consider the stratification of $\mathbb G$ according to the nodal plane sections of $\Delta_5$ with $\delta$ nodes. For this purpose let us define
    \[
        \mathring {\mathbb G}^{\delta} := \{ n \in \mathbb G \ \vert \ \text{ $\mathrm{Sing}(\Delta_N)$ consists of $\delta$ ordinary double points} \}.
    \]
    Clearly, $\mathbb G$ is the disjoint union of the above locally closed strata, for $\delta = 0, \dots, 15$. Indeed, $15$ is the maximal number of nodes for a nodal plane sextic, realized by the union of six general lines. The universal plane section over $\mathring{\mathbb G}^{\delta}$, for every $\delta$, is contained in the pull-back $t^*\Delta_5 \subset \mathbb U$.
    \begin{notation}
        Let us denote by $\mathbb G^{\delta}$ the closure of $\mathring {\mathbb G}^{\delta}$ in $\mathbb G$.
    \end{notation}
    Deformation theory for nodal plane curves, (see, for reference, \cite[Section 4.7]{Sernesi06}), together with the proposition \ref{dixonlemma} imply that:
    
    \begin{thm}\label{mainthm2}
        Every irreducible component of $\mathbb G^{\delta}$ has codimension $\delta$ in $\mathbb G$.
    \end{thm}
    \begin{proof}
        Consider the GIT quotient of $\mathbb G^{\delta}$: by proposition \ref{dixonlemma}, this is a dense locally closed set in $\overline {\mathcal{PS}}^+_{6, \delta}$. The latter is the locus of $\delta$-nodal plane sextics with an even spin structure, in the moduli space $\overline {\mathcal S}^+_{10}$, since a plane sextic curve has arithmetic genus 10. In turn, $\overline {\mathcal{PS}}^+_{6, \delta}$ is a finite covering of the moduli space $\overline {\mathcal P}_{6, \delta}$ of $\delta$-nodal stable plane sextics, via the forgetful map $\chi_+$. This implies that, in the flat family of nodal or smooth curves
        \[
            \mathfrak{\Delta} = \{ \Delta_N = \Delta_5\cap N, \ N \in A \},
        \]
        over an open dense set $A \subset \mathbb G$, the general member $\Gamma$ of $\mathfrak{\Delta}$ over $A^{\delta} := A \cap \mathring {\mathbb G}^{\delta}$ is general in the Severi variety  of $\delta$-nodal curves of its supporting plane $N_{\Gamma}$. Then the points of $\mathrm{Sing} \ \Gamma$ impose $\delta$ independent conditions to the linear system of sextics of $N_{\Gamma}$ passing through  $\mathrm{Sing} \ \Gamma $ (cf. \cite[4.7.2]{Sernesi06}). Hence, the Severi variety of $\delta$-nodal sextics of $N_{\Gamma} $ is smooth at $\Gamma$ and of codimension $\delta$ in $\vert \mathcal O_{N_{\Gamma}}(6) \vert$. Now, since the universal plane $u:\mathbb U \longrightarrow \mathbb G$ is locally trivial, we can assume, up to shrinking $A$ around the point $n_{\Gamma}\in\mathbb{G}$ parametrizing the plane $N_{\Gamma}$, such a trivialization. Then the universal plane $u_A: \mathbb U_{\vert A} \longrightarrow A$ and the trivial bundle $p_1: A \times \mathbb P^2 \longrightarrow A$ are biregularly identified. The same is true, over $A$, for $\mathcal O_{\mathbb U}(1) = t^*\mathcal O_{\mathbb P^{20}}(1)$ and $p_2^* \mathcal O_{\mathbb P^2}(1)$, where $p_2:A\times\mathbb{P}^2\longrightarrow\mathbb{P}^2$ is the second projection. In particular, this biregular map of $\mathbb P^2$-bundles over $A$ identifies the family $\mathfrak \Delta$ of curves $\Delta_N$ over $A$ with a family of plane sextics in $A \times \mathbb P^2$. Thus, applying the previous argument to this family, it follows that its locus $A^{\delta}$ is smooth of codimension $\delta$ at the point $n_{\Gamma}$. Indeed, $A$ is smooth at $n_{\Gamma}$ and the curve $\Gamma$ satisfies the key condition that its singular points impose independent conditions to $\vert \mathcal O_{\mathbb P^2}(6) \vert$.
    \end{proof}

    The next proposition highlights the induced stratification of the Grassmannian $\mathbb{G}_X$.

    \begin{prop}\label{stratification}
        For a general smooth cubic fourfold $X$ there is a stratification of the Grassmannian $\mathbb{G}_X$ in terms of the number of the nodes on the discriminant curve $\Delta_N$. In particular, there exist quasi-projective subvarieties $\mathbb{G}_X^{\delta}\subset\mathbb{G}_X$ parametrizing the nets having discriminant sextic curve with $\delta$ nodes, where $\mathbb{G}_X=\mathbb{G}_X^0\supset \mathbb{G}_X^1\supset\ldots\supset \mathbb{G}_X^9$ and $codim_{\mathbb{G}_X}(\mathbb{G}_X^{\delta})=\delta$.
    \end{prop}

    \begin{rmk}
        It follows from Theorem \ref{mainthm2} that $\mathbb{G}_X^{\delta}=\mathbb{G}^{\delta}\cap\mathbb{G}_X$. In general, every stratum $\mathbb{G}_X^{\delta}\subset \mathbb{G}_X$ contains both nets with discriminant curve with nodes parametrizing rank 4 quadrics and nets with discriminant curve with nodes parametrizing rank 5 quadrics. Indeed, $\Delta_N=N\cap\Delta_5^X$ admits a node at a point if and only if either that point is singular for $\Delta_5^X$ or if $N$ and $\Delta_5^X$ are tangent at that point. In particular, the nodes obtained as tangency points between $N$ and the smooth locus of $\Delta_5^X$ parametrize quadrics of rank 5, while the nodes coming from the intersection between $N$ and $\Delta_4^X$ parametrize quadrics of rank 4. Also, a curve $\Delta_N$ can have both nodes parametrizing rank 5 quadrics and nodes parametrizing rank 4 quadrics.
            
        Moreover, proposition \ref{stratification} points out that a general smooth cubic fourfold $X$ admits nets of polar quadrics with discriminant sextic curve with up to 9 nodes. However, as in the proposition \ref{smoothsextic}, the general net has smooth discriminant curve.
    \end{rmk}

\subsection{The Severi hypersurface $\mathsf{D}_{sn}$ in $\mathbb P^{55}$}\label{Severihypersurface}

    Let $P$ be any projective plane.
    
    \begin{notation}
        We will use the notation $V_{m,\delta}(P)\subset|\mathcal{O}_P(m)|$ for the Severi variety of plane curves of degree $m$ with at least $\delta$ nodes.
    \end{notation}
    
    The Severi variety $V_{m,\delta}(P)$ is the closure of the family of all the $\delta$-nodal curves in $P$ of degree $m$, that is, nodal plane curves $\Gamma$ such that $\vert \mathrm{Sing} \ \Gamma \vert \geq \delta$. The Severi variety is not irreducible. Indeed, it admits an irreducible component $V_{m,\delta}^{irr}(P)$, whose general member is an integral curve. The other components are the images of the various products
    \[
        V_{m',\delta'}^{irr}(P)\times V_{m'', \delta ''}^{irr}(P)
    \]
    in $V_{m,\delta}(P)$ via the sum map, where $m'+m''=m$ and $\delta'+\delta ''=\delta$. Hence, the $V_{m',\delta'}^{irr}(P)$'s correspond to the ways a curve may split up (see \cite{Ran87} and \cite{Ran87a} for a detailed description). The irreducibility of any Severi variety of $\delta$-nodal irreducible plane curves of degree $m$, $V_{m,\delta}^{irr}$, is well known (see \cite{Harris86}). More precisely, the family is empty for $\delta > \binom{m}{2}$, because $\binom{m}{2}$ is the maximal number of nodes, realized by the union of $m$ lines in general position. Otherwise, for $\delta\leq\binom{m}{2}$, its general member $\Gamma$ has exactly $\delta$ nodes and it is irreducible if its geometric genus is $g(\Gamma)\geq 0$.
    
    For the purpose of this work, in what follows we only deal with $V_{6,10}^{irr}(P)$. It is integral of codimension $10$ in $\vert \mathcal O_{\mathbb P^2}(6) \vert$, with general member an irreducible rational sextic with $10$ nodes.
    
    \begin{defn}\label{specialpolarnet}
        A cubic fourfold $X \in\mathbb P^{55}$ admits a \textit{special polar net} $N$ if the discriminant locus $\Delta_N$ is a $10$-nodal plane sextic.
    \end{defn}
    \begin{defn}
        We call Severi locus the set in $\mathbb P^{55}$ of cubic fourfolds with a special polar net. We denote it by $\mathsf{D}_{sn} \subset \mathbb P^{55}$.
   \end{defn}
    A parameter count predicts that $\mathsf{D}_{sn}$ is a hypersurface, and we will prove this in what follows. However, some preliminary remarks are necessary. At first, let us point out that, by Theorem \ref{mainthm2}, the Severi variety $\mathbb G^{10}$ of $\Delta_5$ is equidimensional and such that
   \[
    \dim \mathbb G^{10} = 44 = \dim PGL(6) + \dim \overline {\mathcal{PS}}_{6,10}^+.
    \]
    As already seen in the proof of Theorem \ref{mainthm2}, the quotient $\overline {\mathcal{PS}}_{6,10}^+=\mathbb G^{10} /\!\!/ PGL(6)$ is the coarse moduli space of $10$-nodal plane sextics in $\mathbb P^2$, endowed with an even spin structure. This space dominates the moduli space $\overline{\mathcal P}_{6,10}$ of $10$-nodal plane sextics, via the forgetful map $\chi_+$, which is finite. 
    
    Now, let $n \in \mathbb G$ be the point corresponding to the net $N \subset \mathbb P^{20}$ such that it is general in $\mathbb G^{10}$. Then, we can assume that $\Delta_N$ is general in the Severi variety $V_{6,10}^{irr}(N)$ of the plane $N$. Hence, we can also assume the maximal rank  for  the restriction
    \[
        r: H^0(\mathcal O_{N}(6)) \longrightarrow H^0(\mathcal O_{\mathrm{Sing}\Delta_N}(6)),
    \]
    that is, the surjectivity. Now, consider again the morphisms
    \[\begin{tikzcd}
        & \widetilde {\mathbb G} \arrow[dl,"\widetilde{\gamma}"']\arrow[dr, "\widetilde{\tau}"] &\\
        \mathbb{P}^{55} & & \mathbb{G}.
    \end{tikzcd}\]
    The fiber of $\widetilde \tau$ at the point $n$ is $\mathbb F \times \{ n \}  := \left(\mathbb P^{55} \times \{ n \}\right) \cdot \widetilde{\mathbb G} \subset \mathbb P^{55} \times \mathbb G$. Under our generality assumption on $N$, we can assume that the family of smooth cubic fourfolds  $X \in \mathbb F$ is nonempty. This is indeed an open property over $\mathbb G^{10}$ and it is satisfied over a general net of the Grassmannian  $\mathbb G _{X_{\mathsf f}}$ of the Fermat cubic fourfold $X_{\mathsf f}:=V(\sum_{i=0}^5 x_i^3)$, as we will see later. At first, we want to exclude effective 1-dimensional deformations of $\Delta_N$, in the Grassmannian $\mathbb G_X \subset \mathbb G$, which preserves the ten nodes of $\Delta_N$. We can work for $\mathbb G_X$ exactly as we did for $\mathbb G$ in the proof of Theorem \ref{mainthm2} (cf. proposition \ref{stratification}).
    
    \begin{thm}\label{isolated}
        Let $n$ be the point defined by $N$, then $n$ is isolated in $\mathbb G_X^{10}$.
    \end{thm}
    \begin{proof}
        Let us restrict the universal plane $u:\mathbb U \longrightarrow \mathbb G$ over the Grassmannian \mbox{$\mathbb G_X \subset \mathbb G$}. Then, let us trivialize it along an open neighborhood $A_X$ of $n$. Therefore the family of the plane sections of  $\Delta_5$ over $A_X$ is identified with a family of plane sextics in $A_X \times \mathbb{P}^2$,  with fiber the curve $\Delta_N$ at the point $n \in A_X$. Now, as in the proof of Theorem \ref{mainthm2}, $\dim A_X = 9$ and, moreover, the tangent space to $A_X$ at $n$ is the space of global sections of $\mathcal{O}_{\Delta_N}(1)^{\oplus 3}$. Indeed, $\Delta_N$ is a complete intersection of type $(1,1,1,6)$ in $|J_X|=\mathbb{P}^5$ and its normal bundle is $\mathcal{N}_N:=\mathcal{O}_{\Delta_N}(1)^{\oplus 3}\oplus\mathcal{O}_{\Delta_N}(6)$. In particular, $H^0(\mathcal{O}_{\Delta_N}(1)^{\oplus 3})$ is the space of infinitesimal deformations of the plane supporting $\Delta_N$. As previously done, let us consider, as in \cite{HartshorneHirschowitz85} and \cite{Sernesi84}, the exact sequence
        \[
            \mathcal{T}_{\mathbb{P}^5}\longrightarrow\mathcal{N}_N\longrightarrow T^1_{\Delta_N},
        \]
        defining the Lichtenbaum-Schlessinger sheaf $T^1_{\Delta_N}$. Since $\Delta_N$ is nodal, this corresponds to $\mathcal{O}_{\mathrm{Sing}\Delta_N}$. Moreover, the exact sequence induces on $\mathcal{O}_{\Delta_N}(1)^{\oplus 3}$ a homomorphism
        \[
            h: H^0(\mathcal O_{\Delta_N}(1)^{\oplus 3}) \longrightarrow H^0(\mathcal O_{\mathrm{Sing} \Delta_N}(3)).
        \]
        This is injective, since $N$ is general in $\mathbb G^{10}$, and this implies that $\Delta_N$ is an isolated point in the family $\mathbb G_X^{10}$ of 10-nodal discriminants of the polar nets of $X$.
    \end{proof}
    
    In other words, if we consider a general $\mathbb G_X$ having the property of intersecting $\mathbb G^{10}$ at the point $n$, then, by the injectivity of the previous map $h$, it follows that the plane section $\Delta_N$ is parametrized by a smooth, isolated point in the Severi variety of $10$-nodal plane sections of $\Delta_5^X$. Now, to make a further step in constructing the so-called Severi hypersurface, we consider the pull-back of the universal plane $u: \mathbb U \longrightarrow \mathbb G$ by $\widetilde \tau$:
    \[\begin{tikzcd}
        \widetilde{\mathbb{U}}:=\widetilde{\tau}^*\mathbb{U}\arrow[r]\arrow[d,"\widetilde{u}"] & \mathbb{U}\arrow[d,"u"]\\
        \widetilde{\mathbb{G}}\arrow[r,"\widetilde{\tau}"] & \mathbb{G}.
    \end{tikzcd}\]
    This is a $\mathbb P^2$-bundle over $\widetilde {\mathbb G}$ and the fiber of $\widetilde{u}$ at the point $o := (X,N)\in\widetilde{\mathbb{G}}$ is the projective plane $\widetilde {\mathbb U}_o = N$.
    
    \begin{defn}
        Let us set $\mathbb{S}:=\mathbb{P}(\widetilde{u}_*\mathcal{O}_{\widetilde{\mathbb{U}}}(6))$, where $\mathcal{O}_{\widetilde{\mathbb{U}}}(1):=\widetilde{\tau}^*\mathcal{O}_{\mathbb{U}}(1)$.
    \end{defn}
    
    $\mathbb S$ is a $\mathbb P^{27}$-bundle over $\widetilde {\mathbb G}$. Indeed, if $\sigma: \mathbb S \longrightarrow \widetilde{\mathbb G}$ is its structure map, then the fiber of $\sigma$ at $o = (X,N)\in\widetilde{\mathbb{G}}$ is $\mathbb{S}_o=\mathbb P^{27} = \vert \mathcal O_N(6) \vert$, that is, the linear system of all the sextic curves of the plane $N$.
    
    \begin{defn}
        Let us define the following loci in $\mathbb S$:
        \begin{itemize}
            \item $\mathsf V$ is the locus in $\mathbb S$ whose fiber $\mathsf V_o \subset \mathbb S_o$ at $o = (X,N) \in \widetilde {\mathbb G}$ is the Severi variety $V_{6,10}^{irr}(N)$ of the $10$-nodal irreducible plane sextics of $N$.
            \item $\mathsf N$ is the closure, in $\mathbb S$, of the locus whose fiber $\mathsf N_o \subset \mathbb S_o $ at $o = (X,N) \in \widetilde {\mathbb G}$ is the point parametrizing the curve $\Delta_5^X \cap N$, provided that this intersection is proper.
        \end{itemize}
    \end{defn}
    
    Notice that $\mathsf N$ is the image, in $\mathbb S$, of a rational section $\mathsf n: \widetilde {\mathbb G} \longrightarrow \mathbb S$, hence it is irreducible of dimension $64$. On the other hand, $\mathsf V$ is irreducible too because each of its fibers $\mathsf V_o$ is irreducible of constant dimension. Also, $\mathsf V$ has codimension $10$ in $\mathbb S$.
    
    \begin{defn} 
        Let us set $ \mathsf D := \sigma(\mathsf N \cdot \mathsf V)\subset\widetilde{\mathbb{G}}$.
    \end{defn}
    
    By a dimension count, every irreducible component $\mathsf Z$ of $\mathsf D$ satisfies $\dim \mathsf Z \geq 54$.
    
    \begin{prop}\label{kleiman}
        No irreducible component $\mathsf Z$ of $\mathsf{D}$ dominates $\mathbb P^{55}$ via the projection morphism $\widetilde \gamma: \widetilde {\mathbb G} \longrightarrow \mathbb P^{55}$.
    \end{prop}
    \begin{proof}
        The statement actually follows from Kleiman's work on the transversality of a general translate \cite{Kleiman74}. We have a family of $9$-dimensional Schubert varieties $\mathbb G_X$ in $\mathbb G$, parametrized by $\mathbb P^{55}$. This family maps onto $\mathbb G$ with the morphism $\widetilde \tau$, which is proper and flat over a nonempty open set $A \subset \mathbb G$, as seen in the remark \ref{importantrmk}. Assume, by contradiction, that $\mathsf Z$ dominates $\mathbb P^{55}$ via $\widetilde \gamma$. This implies that, for a general $X \in \mathring{\mathbb P}^{55}$, the Schubert variety $\mathbb G_X$ intersects the closed set $\mathbb G^{10}$, whose codimension in $\mathbb G$ is $10$. This is against Kleiman's result. Indeed, the family of Schubert varieties is $\mathring {\mathbb P}^{55}$ and, by applying \cite[Lemma 1]{Kleiman74}, we should have
        \[
            \dim \mathbb G_X + \dim \mathbb G^{10} - \dim \mathring {\mathbb P}^{55} \geq 0.
        \]
        But we actually have $\dim \mathbb G_X + \dim \mathbb G^{10} - \dim \mathring {\mathbb P}^{55} = -1$, which is a contradiction.
    \end{proof}
    
    Now, Theorem \ref{isolated} implies that there exist irreducible components $\mathsf Z$ of $\mathsf{D}$ having smooth, $0$-dimensional intersection with the fibers of the Grassmann bundle \mbox{$\widetilde \gamma: \widetilde{\mathbb G} \longrightarrow \mathbb P^{55}$}. Notice that, for each irreducible $\mathsf Z$ with this property, the morphism $\widetilde \gamma_{\vert \mathsf Z}$ is generically finite onto the image. This remark, and the proposition \ref{kleiman}, then imply that
    \[
        \dim \widetilde \gamma(\mathsf Z) = \dim \mathsf Z = 54.
    \]
    
    \begin{defn}
        The \textit{Severi hypersurface} $\mathsf D_{sn}$ in $\mathbb P^{55}$ is the union of the above described hypersurfaces $\widetilde{\gamma}(\mathsf Z)$.
    \end{defn}
    
    Notice that the condition for $X \in \mathring{\mathbb P^{55}}$ to have discriminant hypersurface \mbox{$\Delta_5^X \subset \vert J_X \vert$} which contains a $10$-nodal plane sextic is invariant under the action of $PGL(6)$ on the space $\mathbb P^{20} = \vert \mathcal O_{\mathbb P^5}(2) \vert$. This implies the following result.
    
    \begin{thm}
        The closure of the GIT quotient $\mathsf D_{sn}  /\!\!/PGL(6)$ is a divisor in $\mathcal C$.
    \end{thm}
    
    \begin{defn}
        We call the closure of the GIT quotient $\mathsf D_{sn}  /\!\!/PGL(6)$ the \textit{Severi divisor} of $\mathcal C$ and denote it by $\mathcal D_{sn}$.
    \end{defn}

\subsection{Modular constructions}\label{modularconstruction}

    It follows from the previous sections that the Grassmann bundle $\widetilde{\gamma}:\widetilde{\mathbb{G}}\longrightarrow\mathbb{P}^{55}$ induces a Grassmann bundle over the moduli space $\mathcal{C}$ of cubic fourfolds. Indeed, $\widetilde{\gamma}$ is equivariant under the action of $PGL(6)$. We denote this bundle by $\gamma:\mathcal{G}\longrightarrow\mathcal{C}$. More precisely, $\mathcal{G}:=\widetilde{\mathbb{G}}/\!\!/PGL(6)$ and $\gamma^{-1}([X])\cong\mathbb{G}_{[X]}=Gr(3,J_{[X]})$, where $J_{[X]}$ is the vector space of the projective transforms of the polar quadrics of $X$. Finally, $\mathrm{dim} \ \mathcal{G}=\mathrm{dim}\widetilde{\mathbb{G}}-\mathrm{dim}PGL(6)=64-35=29$.
    
    We denote by $\mathcal{N}=\mathbb{G}/\!\!/PGL(6)$ the moduli space of nets of quadric hypersurfaces in $\mathbb{P}^5$, as in section \ref{thedixonlemma}. The morphism $\widetilde{\tau}:\widetilde{\mathbb G}\longrightarrow\mathbb G$ is also equivariant for the action of the group $PGL(6)$; hence, there is an induced modular map $\tau:\mathcal{G}\longrightarrow \mathcal{N}$ that is still well defined and surjective. Moreover, $\mathrm{dim} \ \mathcal{N}=\mathrm{dim}\mathbb{G}-\mathrm{dim}PGL(6)=54-35=19$ and the general fiber of $\tau$ has still dimension 10.

     The next diagram collects what has been analyzed so far:

    \[\begin{tikzcd}\label{diagram}\tag*{$(\diamond)$}
        \mathcal{G}\arrow[d,"\gamma"']\arrow[r,"\tau"] & \mathcal{N}\arrow[d,"\psi"']\arrow[r,"\nu"] & \overline{\mathcal{PS}}_6^+\arrow[dl,"\chi_+"]\\
        \mathcal{C} &\overline{\mathcal{P}}_6 &
    \end{tikzcd}\]
    
    We point out that $\deg \chi_+ = \deg \psi= 2^{10}(2^9 + 1) = 524800$.
    
    \begin{thm}\label{irreduciblecomponentsoftheseveridivisor}
        The Severi divisor $\mathcal{D}_{sn}$ is irreducible.       
    \end{thm}
    \begin{proof}
        Let $P$ be a plane and let us consider the Severi variety $V_{6,10}^{irr}(P)$ and its modular counterpart
        \[
            \mathcal{V}:=V_{6,10}^{irr}(P)/\!\!/PGL(3)\subset\overline{\mathcal{P}}_6\subset\overline{\mathcal{M}}_{10}.
        \]
        Then, via the equivariant maps of the diagram \ref{diagram}, we see that
        \[
            \mathcal{D}_{sn}=\mathsf{D}_{sn}/\!\!/PGL(6)=\gamma(\tau^{-1}(\nu^{-1}(\chi_+^{-1}(\mathcal{V})))).
        \]
        The quotient $\mathcal{V}$ is irreducible since $V_{6,10}^{irr}(P)$ is. The moduli space $\overline{\mathcal{PS}}_6^+$ is irreducible, as a consequence of what was shown by Cornalba in \cite{Cornalba89}, and the map $\chi_+$ is a finite morphism; hence, the preimage $\chi_+^{-1}(\mathcal{V})$ is irreducible. The map $\nu$ is birational, due to the proposition \ref{dixonlemma}, implying that $\nu^{-1}(\chi_+^{-1}(\mathcal{V}))$ is irreducible too. Finally, $\tau$ is a projective bundle with $\mathbb{P}^{10}$ as general fiber so it preserves the irreducibility and this concludes the argument that $\gamma(\tau^{-1}(\nu^{-1}(\chi_+^{-1}(\mathcal{V}))))=\mathcal{D}_{sn}$ is irreducible.
    \end{proof}

    \begin{rmk}\label{fourdivisors}
        The previous constructions were explicitly performed for the irreducible component $V_{6,10}^{irr}(P)$ of the Severi variety, taking into account just the irreducible plane sextic curves with 10 nodes. In particular, a general point in the divisor $\mathcal{D}_{sn}$ parametrizes a cubic fourfold that admits a polar net having an irreducible sextic curve with 10 nodes as discriminant locus. However, what has been done in these pages is not specific to the fact that the discriminant curve $\Delta_N$ is irreducible. Therefore, everything can be repeated also for the curves in the other irreducible components of the variety $V_{6,10}(P)$, giving rise to four other divisors in $\mathcal{C}$. Indeed, the Severi variety $V_{6,10}(P)$ is the union of five irreducible components, consistent with the ways a plane sextic curve with exactly 10 nodes can split up:
        \begin{itemize}
            \item an irreducible 10-nodal sextic;
            \item a smooth cubic and a 1-nodal cubic;
            \item a conic and a 2-nodal quartic;
            \item a line and a 5-nodal quintic;
            \item two lines and a 1-nodal quartic.
        \end{itemize}
        In her PhD thesis \cite{Sammarco25} the author explicitly studies all the five irreducible components of the Severi divisor of $\mathcal{C}$.
    \end{rmk}

\section{The Severi divisor is nonspecial}\label{theSeveridivisorisnonspecial}

    \begin{notation}
        For the sake of compactness we introduce the following notations:
        \begin{itemize}
            \item $\mathsf F := x^3_0 + \ldots + x^3_5$ for the equation of the Fermat cubic fourfold.
            \item $\mathsf f$ for the point defined by $\mathsf F$ in the moduli space $\mathcal C$.
            \item $X_{\mathsf f}:=V(\mathsf F)$ for the Fermat cubic fourfold.
        \end{itemize}
    \end{notation}
    
    To open this section we show that the Fermat cubic fourfold defines a point of the Severi divisor $\mathcal{D}_{sn}$. At first, it is clear that a quadric $Q$ of the polar linear system $\vert J_{X_{\mathsf f}} \vert$ has equation
    \[
        \lambda_0x_0^2+\lambda_1x_1^2+\lambda_2x_2^2+\lambda_3x_3^2+\lambda_4x_4^2+\lambda_5x_5^2=0,
    \]
    where $(\lambda_0: \ldots: \lambda_5)$ are coordinates on $\vert J_{X_{\mathsf f}} \vert := \mathbb P^5$. Then  the discriminant  hypersurface
    \[
        \Delta_5^{X_{\mathsf f}}:=\Delta_5\cap|J_{X_{\mathsf f}}|=  V(\lambda_0\lambda_1\lambda_2\lambda_3\lambda_4\lambda_5) \subset  \mathbb{P}^5
    \]
    is the union of six hyperplanes. Let $N \subset \mathbb P^5$  be a general net of polar quadrics of $X_{\mathsf f}$: then we can assume that $N$ is a plane transversal to $\Delta_5^{X_{\mathsf f}}$, so that its discriminant curve $\Delta_N := \Delta_5^{X_{\mathsf f}} \cap N$ is a nodal union of six lines. We can now show that $\Delta_N$ is in the Severi variety $V_{6,10}^{irr}(N)$.
    
    \begin{thm}\label{limitofirreduciblesextic}
        $\Delta_N$ is limit of a flat family of integral, 10-nodal plane sextics of $N$, that is, $\Delta_N \in V_{6,10}^{irr}(N)$. Moreover $\Delta_N$ is a general union of six lines of $N$.
    \end{thm}
    \begin{proof}
        Let $S \subset \mathbb P^6$ be a smooth sextic Del Pezzo surface, then $S$ is unique up to automorphisms of $\mathbb P^6$ and  it is the image of the rational map $\kappa: \mathbb P^2 \longrightarrow \mathbb P^6$, defined by the linear system $\vert \mathcal I_{\mathsf e}(3) \vert$. Here, $\mathsf e = \lbrace e_1, e_2, e_3 \rbrace \subset \mathbb P^2$ is a set of three non collinear points and $\mathcal I_{\mathsf e}$ is its ideal sheaf. Actually, $S$ is obtained via the blow up $\sigma: S \longrightarrow \mathbb P^2$ of center $\mathsf e$. Moreover, $S$ contains exactly six lines whose formal sum is the hyperplane section
        \[
            E := E_1 + E_2 + E_3 + E_{12} + E_{13} + E_{23},
        \]
        where $E_{ij}$ is the strict transform, via $\sigma$, of the line $\overline { e_i e_j}$ and $E_k = \sigma^{-1}(e_k)$, where $i,j,k \in \lbrace 1, 2, 3 \rbrace$ and $i < j$. We have that $p_a(E) = 1$ and $\mathrm{Sing}(E)$ consists of six nodes:
        \[
            E_1 \cap E_{12}, \ E_1\cap E_{13}, \ E_2 \cap E_{23}, \ E_2 \cap E_{12}, \ E_3 \cap E_{13}, \ E_3 \cap E_{23}.
        \]
        Then $E$ defines a point of the dual variety  $S^{\vee} \subset \mathbb P^{6 \vee}$, that is, the irreducible hypersurface parametrizing the singular hyperplane sections of $S$. A general point in $S^{\vee}$ corresponds to  a 1-nodal irreducible sextic curve. Hence, $E$ moves in an irreducible flat family  
        \[
            \lbrace H_t : t \in T \rbrace \subset S^{\vee},
        \]
        of hyperplane sections of $S$ such that $H_o = E$ for some $o \in T$ and, for $t \neq o$,  the curve $H_t$ is $1$-nodal and integral. Then, $H_t$ is  rational for $t \neq 0$. Now let $P \subset \mathbb P^6$ be a $3$-dimensional linear subspace and let $\pi_P: S \longrightarrow \mathbb P^2$ be the projection from $P$. Then, for a general choice of $P$, we obtain in $\mathbb P^2$ a family of plane sextic curves $ \lbrace \pi_P(H_t) : t \in T \rbrace $ such that a general $\pi_P(H_t)$ is an integral $10$-nodal curve and, for $t = o$, $\pi_P(E)$ is a nodal union of $6$ lines. Hence, $\pi_P(E)$ is limit of $10$-nodal irreducible plane sextic curves, that is, $\pi_P(E)$ is an element of the variety $V_{6,10}^{irr}(\mathbb{P}^2)$ of integral and $10$-nodal sextics. \par
        We also claim that a general $\pi_P(E)$ is projectively equivalent to a general $\Delta_N$. Indeed, the family of nodal unions of six lines in $\mathbb P^2$ is irreducible and its GIT quotient $\mathsf Q$ is $4$-dimensional. Now, we have $E = \mathbb P^5 \cap S$ and the family of  projections of $E$ in $\mathbb P^2$ dominates $\mathsf Q$. Indeed, this family is parametrized by the $9$-dimensional Grassmannian $Gr(3,H^0(\mathcal O_E(1)))$. Moreover, its image in $\mathsf Q$ is $Gr(3,H^0(\mathcal O_E(1)))/\!\!/\Gamma$, where $\Gamma$ is the stabilizer of $E$ in $\mathrm {Aut}\langle E \rangle$, where $\langle E \rangle \subset \mathbb P^6$ is the hyperplane spanned by $E$. It is easy to see that $\Gamma$ is $5$-dimensional. Hence $\dim Gr(3,H^0(\mathcal O_E(1)))/\!\!/\Gamma = 4$.
    \end{proof}
    
    Let $(\Delta_N, \theta)$ be an even spin curve such that $\Delta_N$ is a general plane section of the discriminant hypersurface of $\vert J_{X_{\mathsf f} }\vert$. The next result follows as a corollary.
    
    \begin{cor}
        $(\Delta_N, \theta)$ moves in a flat, integral family $\lbrace (\Delta_t, \theta_t), t \in T \rbrace$ of even spin curves where, for a general $t$, the stable model of $\Delta_t$ is an integral $10$-nodal plane sextic. 
    \end{cor}
    \begin{proof}
        In $\overline {\mathcal M}_{10}$ let us take the irreducible curve defined by the family of plane sextics with $10$ nodes, $\lbrace [\pi_P(H_t)] : t \in T\rbrace$, constructed in the proof of Theorem \ref{limitofirreduciblesextic}. Then, let us lift this curve via the forgetful map $\chi_+:\overline{\mathcal{S}}_{10}^+\longrightarrow \overline {\mathcal M}_{10}$. Since $\chi_+$ is finite, every point $[\Delta_N,  \theta]$ belongs to this lift. The curves $\Delta_t$ are the preimages of the $\pi_P(H_t)$'s via $\chi_+$.
    \end{proof}

    We will build a family of tangent lines to the Severi variety $V_{6,10}^{irr}(N) \subset \vert \mathcal O_N(6) \vert$ at its element $\Delta_N$. Let $\Delta_N$ be a general union of six lines, then we can write
    \[
        \Delta_N =  \mathscr T \cup \overline {\mathscr T},
    \]
    where $\mathscr T$, $\overline {\mathscr T}$ denote two triangles of lines in $\Delta_N$ whose union is the curve $\Delta_N$. Now, we consider the ideal sheaves $\mathcal I_{\mathscr T}$ and $\mathcal I_{\overline {\mathscr T}}$ of $\mathscr T$ and $\overline {\mathscr T}$ and we finally define the subspaces
    \[
        \mathbb P^9_{\mathscr T} := \vert \mathcal I_{\mathscr T}(6) \vert \ \text {and} \  \mathbb P^9_{\overline {\mathscr T}} := \vert \mathcal I_{\overline {\mathscr T}}(6) \vert
    \]
    of $\vert \mathcal O_N(6) \vert := \mathbb P^{27}$. These are the $9$-dimensional linear systems of sextics with fixed component respectively $\mathscr T$ and $\overline {\mathscr T}$. Their union spans an 18-dimensional linear space which we denote by $\mathbb P^{18}_{\mathscr T, \overline {\mathscr T}}$. We have
    \[ 
        \mathbb P^9_{\mathscr T} \cup  \mathbb P^9_{\overline {\mathscr T}} \ \subset \  \mathbb P^{18}_{\mathscr T, \overline {\mathscr T}} \subset \mathbb P^{27}.
    \]

    \begin{thm}\label{triangle+cubic}
        The subspace $\mathbb P^{18}_{\mathscr T, \overline {\mathscr T}}$ is tangent  to $V_{6,10}^{irr}(\mathbb{P}^2)$ at $\Delta_N$.   
    \end{thm}
    \begin{proof}
        Let $o \in Y \subset \mathbb P^N$, where $Y$ is a projective variety. We recall that any union of projective linear subspaces contained in $Y$, and passing through $o$, spans a subspace contained in the tangent projective subspace to $Y$ at $o$. We apply this property to our situation. For $i = 1,2,3$ we consider in $\mathbb P^{27}$ the 6-dimensional linear systems of curves $\mathscr T \cup \mathscr F$ such that $\mathscr F \subset N$ is a cubic singular at the vertex $\overline v_i$ of the triangle $\overline {\mathscr T}$. We denote these by
        \[
            \mathbb{P}^6_{\mathscr T, i} \subset \mathbb P^{27}, \ \ i = 1,2,3.
        \]
        These linear systems are in $\mathbb{P}^9_{\mathscr T}$ and span it, as immediately follows from the equality  
        \[
            \mathbb P^6_{\mathscr T,1} \cap \mathbb P^6_{\mathscr T, 2} \cap \mathbb{P}^6_{\mathscr T,3} = \left\{\mathscr T \cup \overline{\mathscr T}\right\}.
        \]
        Now, as in the proof of Theorem \ref{limitofirreduciblesextic}, the projections in $\mathbb P^2$ of the hyperplane sections of a sextic Del Pezzo surface $S$ can be used to show that $\mathbb P^6_{\mathscr T, i} \subset V_{6,10}^{irr}(\mathbb P^2)$. Indeed, a general element of $\mathbb P^6_{\mathscr T, i}$ is $\mathscr T \cup \mathscr F$, where $\mathscr F$ is a 1-nodal cubic, and we can show that $\mathscr T \cup \mathscr F \in V_{6,10}^{irr}(\mathbb{P}^2)$: in the proof of Theorem \ref{limitofirreduciblesextic} it suffices to replace the hyperplane section $E$ with the one cut out on $S$ by a general $\mathbb{P}^5$ passing through $E_i+E_j+E_{ij}$. This implies that $\mathbb{P}^6_{\mathscr T,i} \subset V_{6,10}^{irr}(\mathbb{P}^2)$ and hence the statement.
    \end{proof}
    
    The next corollary is immediate. 

    \begin{cor}\label{corollary4.4}
        Let $L\subset \vert\mathcal O_N(6) \vert$ be a pencil containing $\Delta_N$ and having a triangle $\mathscr T$ as a fixed component. Then $L$ is a tangent line to $V_{6,10}^{irr}(N)$ at the element $\Delta_N$.
    \end{cor}

    \begin{rmk}\label{normalizationoftheSeverivariety}
        The Severi variety $V_{6,10}^{irr}(\mathbb P^2)$ is not smooth at $\Delta_N$, as can also be deduced from Theorem \ref{triangle+cubic}. Indeed, its dimension is 17 while its projective tangent space at $\Delta_N$ contains a linear space of dimension 18. The curve $\Delta_N$ has 15 nodes, since it is the union of six lines, and choosing 10 of these results in a smooth branch of the normalization of $V_{6,10}^{irr}(\mathbb P^2)$ at the point corresponding to $\Delta_N$.
    \end{rmk}

    Now, the interest of the Severi divisor lies in the following theorem.

    \begin{thm}\label{maintheorem}
        The divisor $\mathcal{D}_{sn}\subset\mathcal{C}$ is not a Noether-Lefschetz divisor.
    \end{thm}
    
    The structure of the proof relates to that of \cite[Proposition 4.16]{RanestadVoisin17}. We rely on the next statement, referring the reader to \cite[Lemma 4.17]{RanestadVoisin17} and \cite[Corollary 3.3]{Voisin13} for more information.
    
    \begin{prop}
        In a local universal family of deformations of cubic fourfolds, Noether-Lefschetz divisors have a smooth normalization.
    \end{prop}
    
    Therefore, to show that $\mathcal{D}_{sn}$ is not a Noether-Lefschetz divisor, our first need is to build, at a suitable point $\mathsf o \in \mathcal D_{sn} \subset \mathcal C$, the local universal family of deformations, or Kuranishi family, for $\mathsf F_o$, where $\mathsf F_o \in {\rm Sym}^3 V_6$ defines the point $\mathsf o$ in  $\mathcal C$. Then, studying such a family, we have to deduce that the normalization of $\mathcal D_{sn}$ at $\mathsf o$ is not smooth.
    
    We perform this program for the point $\mathsf f$ defined by the Fermat equation $ \mathsf F$. As in the case of \cite[Proposition 4.18]{RanestadVoisin17}, we construct a suitable set of $20$ local analytic deformations of $\mathsf F$ in the divisor $\mathcal D_{sn}$, in an open analytic disk around $\mathsf f \in \mathcal C$. The elements of this set are 1-dimensional and their tangent vectors at $\mathsf f$ naturally define a basis for the $20$-dimensional tangent space to $\mathcal C$ at $\mathsf f$. Then this basis  will be used to address the non smoothness of the normalization of $\mathcal D_{sn}$ at $\mathsf f$.
    
    Preliminarily, let us recall the construction, in the affine space $\mathrm{Sym}^3V_6$, of the tangent space to the orbit $Gl(V_6)\cdot \mathsf F$ of $\mathsf F$ (cf. \cite[Remark 6.16]{Voisin03}). As it is well known, this is the space
    \[
        \mathsf F + J_{X_{\mathsf f}}[3],
    \]
    where $J_{X_{\mathsf f}}[3]$ denotes the space of degree $3$ forms in the Jacobian ideal 
    \[
        J_{X_{\mathsf f}}= \ <\frac{\partial \mathsf F}{\partial x_0},\ldots,\frac{\partial \mathsf F}{\partial x_5}> \ = \ <x_0^2,\ldots,x_5^2>.
    \]
    Hence, the next natural step is to consider in ${\rm Sym}^3 V_6$ the affine space $\mathsf F + T$, where
    \[
        T := \langle x_ix_jx_k, \ 0 \leq i < j < k \leq 5 \rangle
    \]
    is the $20$-dimensional vector space generated by the monomials $x_ix_jx_k$. Clearly the affine subspaces $\mathsf F + J_{X_{\mathsf f}}[3]$ and $\mathsf F + T$ are complementary in ${\rm Sym}^3 V_6$, which implies that $\mathsf F + T$ is transverse to the $Gl(V_6)$-orbit of  $\mathsf F$ at $\mathsf F$. Then, as in \cite{RanestadVoisin17}, let us consider the morphism
    \begin{align*}
        \phi: Gl(V_6)\times T&\longrightarrow\mathrm{Sym}^3V_6\\
        (\gamma,t)&\longmapsto\gamma(\mathsf F)+t.
    \end{align*}
    Its general fiber is bijective to the stabilizer of $\mathsf F$ in $Gl(V_6)$. As it is well known, $\phi$ is dominant and its differential at $(id,0)$ is an isomorphism. Moreover, there exists an open analytic neighborhood $A' \subset\mathrm{Sym}^3V_6$ of $\mathsf F$ and a holomorphic retraction of $\phi$, say $\pi:A' \longrightarrow A \subset T$, such that $\lbrace \pi(v) \rbrace = A \cap O_v$, where $O_v$ is the $Gl(V_6)$-orbit of $v$ (cf. \cite{RanestadVoisin17} and \cite[Chapter 6]{Voisin03}). This construction provides the local universal deformation, or Kuranishi family, over the open analytic neighborhood $A$ of $\mathsf F$ in $T$ (see \cite{Sernesi06} and \cite{Kuranishi62}). This family is obtained by restricting the universal hypersurface in $\mathrm{Sym}^3V_6\times\mathbb{P}^5$ to $A \times\mathbb{P}^5$. Relying on this setting, it is standard to identify $T$ with the space of infinitesimal deformations of $\mathsf F$ and to construct a useful basis for it. Indeed it suffices to consider in $\mathrm{Sym}^3V_6$ the set of affine lines:    
    \[
        S=\{F_{ijk}(t):=tx_ix_jx_k+x_0^3+x_1^3+x_2^3+x_3^3+x_4^3+x_5^3 \ | \ t\in\mathbb{C}, \ i < j <  k\}.
    \]
    Each line $F_{ijk}(t)$ is a deformation of $\mathsf F$ and the vector $\left(\frac{\partial F_{ijk}(t)}{\partial t}\right)_0 = x_ix_jx_k$ generates its tangent space at $\mathsf F$. Taking the quotient of $\mathrm {Sym}^3 V_6$ by $J_{X_{\mathsf f}}[3]$, we can fix the identity
    \[
        \mathrm {Sym}^3 V_6 / J_{X_{\mathsf f}}[3] = T.
    \]
    Clearly, the set of vectors $\left\{\left(\frac{\partial F_{ijk}(t)}{\partial t}\right)_0, \ i < j < k \right\}$ is a basis of $T$ and the local universal family over $A$ defines an injective homomorphism  $\mathbf{T}_{\mathsf F} A \longrightarrow \mathbf{T}_{\mathsf f}\, \mathcal C$ of tangent spaces. It is important to recall that this injective homomorphism is actually an isomorphism, because the moduli space $\mathcal C$ is smooth and $20$-dimensional.
    
    \begin{rmk}\label{automorphismsoftheFermat}
        An important remark is that the stabilizer of $\mathsf F$, $Stab(\mathsf F)$, acts on $T$ leaving it unchanged. Indeed, $Stab(\mathsf F)$, which in the case of the Fermat cubic fourfold coincides with $\rm Aut(X_{\mathsf f})$, is $(\mu_3)^5 \rtimes \mathfrak S_6$, where $\mu_3$ is the group of cubic roots of unity, and the semidirect product structure is given by the permutation of the coordinates by the symmetric group $\mathfrak S_6$. An element of the diagonal subgroup of $Stab(\mathsf F)$ has the form
        \[
            \gamma = \mathrm{diag}(\zeta_0, \dots, \zeta_5), \qquad \zeta_i^3 = 1,\quad \prod_{i=0}^5 \zeta_i = 1,
        \]
        acting on monomials by
        \[
            \gamma \cdot (x_i x_j x_k) = (\zeta_i \zeta_j \zeta_k)\, x_i x_j x_k.
        \]
        Since this is a scalar multiple of the same monomial, it lies again in $T$. Thus the diagonal subgroup preserves $T$. An element $\sigma \in \mathfrak S_6$ acts by permuting coordinates
        \[
            \sigma \cdot (x_i x_j x_k)= x_{\sigma(i)}\, x_{\sigma(j)}\, x_{\sigma(k)},
        \]
        where the indices $\sigma(i)$, $\sigma(j)$, $\sigma(k)$ are still distinct, so the resulting monomial is again a generator of $T$. Hence, the entire group $Stab(\mathsf F)$ preserves $T$, which is a $Stab(\mathsf F)$-stable subspace of $\mathrm{Sym}^3V_6$. Then, by composing $\pi$ by $g \in (\mu_3)^5 \rtimes \mathfrak S_6$, we obtain another holomorphic retraction $g \circ \pi: A' \longrightarrow g(A) \subset T$.
    \end{rmk}

    We can now prove the theorem.

    \begin{proof}[Proof of Theorem \ref{maintheorem}]
        Now we concentrate on the Severi divisor $\mathcal D_{sn}$, aiming to show that its normalization at $\mathsf f$ is not smooth. Therefore, via the local universal family at $\mathsf f$, we study $\mathcal D_{sn}$ in an open analytic disk around $\mathsf f$. Let $\mathcal D_{A'} \subset A'$ be the restriction to $A'$ of the pull-back of $\mathcal D_{sn}$ by the quotient map ${\rm Sym^3} V_6 \longrightarrow \mathcal C$. Then we consider its retraction in $A$, that is, $\mathcal D_A := \pi(\mathcal D_{A'})$. In other words, locally, the divisor comes from a divisor $\mathcal{D}_A\subset A$ such that its image $\mathcal{D}_{sn}$ in $\mathrm{Sym}^3V_6/\!\!/Gl(V_6)$ is obtained by taking the quotient of $\mathcal{D}_A$ by $Stab(\mathsf F)  \cong (\mu_3)^5 \rtimes \mathfrak S_6$, which is nontrivial. For this reason, we cannot identify here $\mathcal{D}_A$ with an open set of $\mathcal{D}_{sn}$.
    
        If $\mathcal{D}_{sn}$ were a Noether-Lefschetz divisor, then the divisor $\mathcal{D}_A$ in the local universal family of deformations should have smooth local branches, otherwise the normalization of the Severi divisor would still have singular local branches. Indeed, to obtain the normalization of the Severi divisor and to study it locally at the points coming from $\mathsf f$, first we need to eliminate the singularities coming from the automorphisms of $\mathsf F$ and then we can separate the branches. Hence, the divisor $\mathcal{D}_A$ represents the intermediate step towards to the normalization of $\mathcal{D}_{sn}$, in which we have solved the singularities caused by the nontriviality of the stabilizer of $\mathsf F$.
        
        We consider the 20 affine lines, which are the elements of the  set
        \[
            S = \left\{ F_{ijk}(t)\ \vert \ i <j < k \right\} \subset \mathsf F + T \subset {\rm Sym}^3 V_6.
        \]
        We observe that, for any $t$, the set of quadratic forms
        \[
            \left\{\frac{\partial F_{ijk}(t)}{\partial x_l}=\begin{cases}
                3x_l^2 \ &\mathrm{if} \ l\not\in\{i,j,k\}\\
                \frac {tx_ix_jx_k}{x_l} +3x_l^2 \ &\mathrm{if} \ l\in\{i,j,k\}
            \end{cases}\ \middle| \ l=0,\ldots,5 \right\}
        \]
        is a basis for the vector space defining the polar linear system $\vert J_{X_{ijk}}(t) \vert$ of the fourfold $X_{ijk}(t):=V(F_{ijk}(t))$. Since the symmetric group $\mathfrak S_6$ acts transitively on the set $S$ of pencils, it will not be restrictive to fix $i,j,k$. Then, we assume $i=1, \ j=3, \ k=5$ and consider $F_{135}(t) = tx_1x_3x_5+\mathsf F$. In this case the above basis of quadratic forms is
        \[
            \left\{ 3x_0^2, \ tx_3x_5+3x_1^2, \ 3x_2^2, \ tx_1x_5+3x_3^2, \ 3x_4^2, \ tx_1x_3+3x_5^2 \right\}.
        \]
        Then we fix coordinates $(y_0:\ldots:y_5)$ on $\vert J_{X_{135}(t)} \vert$, for any $t$, so that the equation of a general polar quadric is
        \[
            3y_0x_0^2+ty_1x_3x_5+3y_1x_1^2+3y_2x_2^2+ty_3x_1x_5+3y_3x_3^2+3y_4x_4^2+ty_5x_1x_3+3y_5x_5^2=0.
        \]
        The associated matrix is then
        \[M(t)=\begin{pmatrix}
            3y_0 & 0 & 0 & 0 & 0 & 0\\
            0 & 3y_1 & 0 & \frac{t}{2}y_5 & 0 & \frac{t}{2}y_3\\
            0 & 0 & 3y_2 & 0 & 0 & 0\\
            0 & \frac{t}{2}y_5 & 0 & 3y_3 & 0 & \frac{t}{2}y_1\\
            0 & 0 & 0 & 0 & 3y_4 & 0\\
            0 & \frac{t}{2}y_3 & 0 & \frac{t}{2}y_1 & 0 & 3y_5
        \end{pmatrix}\]
        and the discriminant locus $\Delta_5^{X_{135}(t)} \subset \vert J_{X_{135}(t)} \vert$ is defined by the equation
        \[
            \mathrm{det}M(t)=27y_0y_2y_4\left[\left(27+\frac{t^3}{4}\right)y_1y_3y_5-\frac{3}{4}t^2\left(y_1^3+y_3^3+y_5^3\right)\right]=0.
        \]
        Using the polynomial $\mathrm{det}M(t)$ and the line $F_{135}(t) := \mathbb A^1$ let us define the  morphism
        \begin{align*}
            \widehat \alpha: \mathbb A^1 &\longrightarrow \mathrm{Sym}^6 V_6\\
            t&\longmapsto\det M(t).
        \end{align*}
        Let $\mathbb V \subset \mathrm{Sym}^6 V_6$ be the $2$-dimensional subspace whose elements are the vectors
        \[
            27y_0y_2y_4 \left[v_0(y_1y_3y_5) - v_1(y_1^3+y_3^3+y_5^3)\right],
        \]
        where $(v_0,v_1)$ are the affine coordinates on $\mathbb V$. Then we have $\widehat \alpha(\mathbb A^1) \subset \mathbb V$, where
        \[
            v_0 =  27+\frac{t^3}4 \ , \ v_1 = \frac{3}{4}t^2.
        \]
        Passing to the projective space $\mathbb P(\mathbb V)$, we fix coordinates $(v_0:v_1)$ on it. Clearly, $\mathbb P(\mathbb V)$ is embedded as a line in $\vert \mathcal O_{\mathbb P^5}(6) \vert$ and coincides with the previous pencil of sextic hypersurfaces, having $V(y_0y_2y_4)$ as its fixed component and $V(y_0y_1y_2y_3y_4y_5)$ among its members.  By composing  $\widehat \alpha$ with the quotient map $\mathbb V - \lbrace \overline 0 \rbrace \longrightarrow \mathbb P(\mathbb V)$, we can define the morphism
        \begin{align*}
            \alpha: \mathbb A^1 &\longrightarrow \mathbb P(\mathbb V) \subset\vert \mathcal O_{\mathbb P^5}(6) \vert\\
            t&\longmapsto\left(27+\frac{t^3}4 \ \colon \ \frac{3}{4}t^2\right).
        \end{align*}
        Putting $u := \frac {v_1}{v_0}$, the morphism $\alpha$ is defined on $\mathbb A^1 - V\left(27+ \frac{t^3}4\right)$ by the rational function
        \[\label{equation}\tag{$\star$}
            u = \frac{3t^2}{108+t^3}.
        \]
        Now, consider $\mathbb F:=\lbrace t_0 \mathsf F + t_1x_1x_3x_5\ \vert \ (t_0,t_1) \in \mathbb C^2 \rbrace \subset \mathrm {Sym}^3V_6$. Putting $t = \frac {t_1}{t_0}$, the projective line $\mathbb P(\mathbb F)$ is the completion of the affine line $F_{135}(t)$. Obviously $\alpha$ extends to $\alpha:\mathbb P(\mathbb F) \longrightarrow \mathbb P(\mathbb V)$, which is a finite triple covering.
        
        Now, we fix a plane $N$ in the space $\mathbb P^5$ with coordinates $(y_0:\ldots:y_5)$, and the relative restriction map $\rho_N: \vert \mathcal O_{\mathbb P^5}(6) \vert \longrightarrow \vert \mathcal O_N(6) \vert$, such that $N$ is general and transversal to the morphism $\alpha$. Then, we define $\alpha_N: \mathbb A^1 \longrightarrow  \vert \mathcal O_N(6) \vert$ to be the composition $\rho_N \circ \alpha$. It is clear that $\alpha_N (\mathbb A^1)$ is an affine pencil of plane sextics of $N$, having the triangle of lines
        \[
            \mathscr T_N := V(y_0y_2y_4) \cap N
        \]
        as a fixed component. Let $\Delta_N(t) := V( \mathrm{det} M(t)) \cap N$; then $\Delta_N(0)$ is a union of six lines. Since $N$ is general, both $\mathscr T_N$ and $\Delta_N(0)$ are general nodal curves in their respective families of union of lines. Hence, it follows from corollary \ref{corollary4.4} that the affine line $\alpha_N(\mathbb A^1)$ is tangent to the Severi variety $V_{6,10}^{irr}(N)$ at its point $\alpha_N(0) = \Delta_N(0)$.

        The previous considerations and constructions are valid for any choice of an ordered triple $(i,j,k)$, that is for every affine line $F_{ijk}(t)$ in the set $S$. Now, we show that any $F_{ijk}(t)$ is tangent to $\mathcal D_A$ at $\mathsf F$ and hence the following lemma.

        \begin{lem}
            The Severi divisor $\mathcal D_{sn}$ is singular at $\mathsf f \in \mathcal C$.
        \end{lem}
        \begin{proof}
            A point of the vector space $T$ is a linear combination $\sum m_{ijk}x_ix_jx_k$, hence it can be defined by the lexicographically ordered 20-tuple $\overline{m}:=(m_{ijk})_{0\leq i<j<k\leq 5}$. Then a point in the affine space $\mathsf F + T\subset\mathrm{Sym}^3V_6$ will be denoted by $\mathsf F+\overline{m}:=\mathsf F+\sum m_{ijk}x_ix_jx_k$. Recall that $A\subset T$ is the open analytic neighborhood of $\mathsf F$ previously constructed through a holomorphic retraction and consider the map $\delta_A:A\longrightarrow \vert \mathcal O_N(6) \vert$ sending $\overline{m} \in A$ to the unique discriminant sextic curve of the net $N$ of polar quadrics of $\mathsf F +\overline{m}$. Then the pull-back by $\delta^*_A$ of $V_{6,10}^{irr}(N)$ is contained in $\mathcal{D}_A$, since the latter is the retraction of the pull-back of $\mathcal{D}_{sn}$ by the quotient map $\mathrm{Sym}^3V_6\longrightarrow\mathcal{C}$. Now, with the same notations as before, observe that $\delta_A$ restricts to $\alpha_N$ along $A \cap F_{ijk}(t)$ and that the equality (\ref{equation}) implies that $\alpha_N$ is ramified at $t=0$, that is, at $\mathsf F$. This implies that $F_{ijk}(t)$ is a tangent line to $\mathcal{D}_A$.  Since the tangent vectors to the 20 lines $F_{ijk}(t)$ form a basis of $T$, it follows that $T\subset \mathbf{T}_{\mathsf F}\mathcal{D}_A$ and, for dimensional reasons, that $T$ is actually the Zariski tangent space to $\mathcal{D}_A$ at the point $\mathsf F$. Again by a dimension count it follows that $\mathcal{D}_A$ is singular at $\mathsf F$, that is, $\mathsf f$ is singular for $\mathcal D_{sn}$.
        \end{proof}

        Now there is one last step left to prove that $\mathcal{D}_{sn}$ is nonspecial. The previous lemma implies that the general line contained in $\mathsf F+T$ passing through the Fermat polynomial has multiplicity 2 in $\mathsf F$, by a semicontinuity argument. Then, the tangent cone to $\mathcal{D}_A$ at $\mathsf F$ is the union of the lines of $\mathsf F+T$ passing through $\mathsf F$ with multiplicity at least 3 and it is denoted by $\mathbf{C}_{\mathsf F}(\mathcal D_A)$.
        
        \begin{defn}
            Let $X$ be an algebraic variety, $x$ a point of $X$, and $(\mathcal{O}_{X,x}, \mathfrak{m})$ the local ring of $X$ at $x$. The \textit{tangent cone} $\mathbf{C}_xX$ to $X$ at $x$ is $\mathrm{Spec}(\rm gr_{\mathfrak m}\mathcal{O}_{X,x})$, where
            \[
                \rm gr_{\mathfrak m}\mathcal{O}_{X,x}=\bigoplus_{i\geq 0}\mathfrak{m}^i/\mathfrak{m}^{i+1}.
            \]
        \end{defn}
        In our case,
        \[
            \left(\mathcal{O}_{\mathcal{D}_A,\mathsf{F}},\mathfrak{m}\right)=\left(\left(\frac{\mathbb{C}[m_{012},\ldots,m_{345}]}{(f)}\right)_{(m_{012},\ldots,m_{345})},(m_{012},\ldots,m_{345})\right),
        \]
        where $f=f(m_{012},\ldots,m_{345})$ is the local equation of $\mathcal{D}_A$ at the point $\mathsf{F}$ and $(m_{ijk})_{0\leq i<j<k\leq 5}$ are the coordinates of $A$. Let us consider the Taylor expansion of $f$ at $\mathsf{F}$,
        \[
            f=f_0+f_1+f_2+f_3+\ldots
        \]
        where $f_i=f_i(m_{012},\ldots,m_{345})$, $\mathrm{deg}(f_i)=i$ and $f_i\in\mathfrak{m}^i/\mathfrak{m}^{i+1}$. Since the divisor $\mathcal{D}_A$ is singular at the Fermat and since the expansion is centered at the origin of $A$, we have that $f_0\equiv f_1\equiv 0$. Then,
        \[
            \mathbf{C}_{\mathsf F}(\mathcal D_A)=\{(m_{012},\ldots,m_{345})\in A \ \vert \ f_2(m_{012},\ldots,m_{345})=0\}\subset T
        \]
        is the tangent cone, which is the zero locus in $T$ of a homogeneous polynomial of degree 2. As a consequence of the remark \ref{automorphismsoftheFermat} and by construction of $\mathcal{D}_A$, it follows that $\mathbf{C}_{\mathsf F}(\mathcal D_A)$ is invariant under the action of the stabilizer of the Fermat polynomial. To explicitly write its equation we need to find all the quadric hypersurfaces of $T$ that are invariant with respect to the action of $(\mu_3)^5 \rtimes \mathfrak S_6$. Let
        \[
            Q=\sum_{\substack{0\leq i<j<k\leq 5 \\ 0\leq a<b<c\leq 5}} c_{(i,j,k,a,b,c)}\, m_{ijk}m_{abc}\in\mathrm{Sym}^2T
        \]
        be a general quadric in the variables of $T$. With the same notations of the remark \ref{automorphismsoftheFermat}, an element $g\in(\mu_3)^5\rtimes \mathfrak S_6$ acts on $Q$ in the following way:
        \[
            g \cdot Q = \sum_{\substack{0\leq i<j<k\leq 5 \\ 0\leq a<b<c\leq 5}}c_{(i,j,k,a,b,c)}\, (\zeta_i \zeta_j \zeta_k \zeta_a \zeta_b \zeta_c)\, m_{\sigma(i)\sigma(j)\sigma(k)}m_{\sigma(a)\sigma(b)\sigma(c)}.
        \]
        It is immediate to see that $Q=g\cdot Q$, for every $g\in(\mu_3)^5\rtimes \mathfrak S_6$, if and only if:
        \[
            c_{(i,j,k,a,b,c)}=c_{(\sigma(i),\sigma(j),\sigma(k),\sigma(a),\sigma(b),\sigma(c))} \quad \forall \sigma, \, \forall (i,j,k),(a,b,c)
        \]
        and
        \[
            \zeta_i \zeta_j \zeta_k \zeta_a \zeta_b \zeta_c=1, \quad \forall (i,j,k),(a,b,c).
        \]
        But these two conditions hold if and only if the quadric $Q$ is, up to multiplication by a scalar,
        \[
            Q=\sum\limits_{\substack{0\leq i<j<k\leq 5 \\ \{a,b,c\}=\{0,\ldots,5\}\setminus\{i,j,k\}}}m_{ijk}m_{abc}\, ,
        \]
        that is the sum of the 10 mixed products where the triples $(i,j,k)$ and $(a,b,c)$ are complementary. Its associated matrix has rank 20 since it is anti-diagonal with all the elements on the secondary diagonal equal to 1. Hence, $Q$ defines a smooth irreducible quadric in $T$ that is necessarily equal to the tangent cone $\mathbf{C}_{\mathsf F}(\mathcal{D}_A)$. Then, $f_2$ is a non-degenerate quadric and the only point in which $\nabla f_2=0$ is the origin of $A$, that is, the point $\mathsf F\in\mathcal D_A$. Hence, $\mathsf F$ is an isolated singular point of type $A_1$ for the divisor $\mathcal D_A$. It follows from \cite[II.8.23]{Hartshorne77} that the local ring $\mathcal{O}_{\mathcal{D}_A,\mathsf{F}}$ is normal and hence the normalization of $\mathcal{D}_A$ is an isomorphism in a neighborhood of $\mathsf{F}$. This shows that the normalization of the Severi divisor $\mathcal{D}_{sn}$ at $\mathsf f$ is not smooth implying that $\mathcal{D}_{sn}$ is not of Noether-Lefschetz type.

    \end{proof}

    \begin{rmk}
        Theorem \ref{maintheorem} implies that, for a very general cubic fourfold of the Severi divisor $\mathcal{D}_{sn}$, there is no nonzero Hodge class in $H^4(X,\mathbb{Q})_{prim}$. Indeed, it follows from \cite{Voisin13} that the Hodge loci distribute in a countable union of hypersurfaces of $\mathcal C$, as a consequence of the fact that $h^{3,1}(X)=1$ is an upper bound for the codimension of each irreducible component of these loci.
        
        Another important remark, that completes this description of the locus of cubic fourfolds having a special polar net, is the fact that the other four divisors of $\mathcal{C}$, defined on the basis of the other four irreducible components of the Severi variety $V_{6,10}(\mathbb{P}^2)$ as mentioned in the remark \ref{fourdivisors}, are also nonspecial. This can be shown with techniques similar to those used in the proof of Theorem \ref{maintheorem}. Hence, this gives five nonspecial irreducible divisors in the moduli space $\mathcal C$, thus obtaining new nonspecial divisors with respect to the three known and described so far in the literature.
    \end{rmk}

\end{document}